\documentclass[12pt]{article}
\usepackage{amsmath,amsthm, amssymb}
\usepackage{mathabx,tikz}
\usepackage[all]{xy}

\usepackage[active]{srcltx}
\newtheorem{theorem}{Theorem}[section]
\newtheorem{lemma}[theorem]{Lemma}

\theoremstyle{definition}
\newtheorem{definition}[theorem]{Definition}

\renewcommand{\ge}{\geqslant}
\renewcommand{\dashrightarrow}
{\text{\raisebox{0.9mm} {\
\begin{tikzpicture}[->,thick,xscale=0.56]
  \draw[dashed] (0,0)--(1,0)
;
\end{tikzpicture}}\ }}

\newcommand{\dash}{\xymatrix{\ar@{-->}[r]&}}
\newcommand{\full}{\xymatrix{\ar@{->}[r]&}}

\newcommand{\ci}{
\begin{picture}(6,6)
\put(3,3){\circle*{3}}
\end{picture}}
\newcommand{\lin}{\,\frac{}{\quad}\,}

\begin{document}

\title{Tame systems of linear and semilinear mappings}

\author{\emph{Tatiana Klimchuk, Dmitry Kovalenko}\\
Kiev National Taras Shevchenko
University, Kiev, Ukraine\\
klimchuk.tanya@gmail.com\quad kovalenko.d.y@gmail.com
 \and
\emph{Tetiana Rybalkina, Vladimir V. Sergeichuk}\\
Institute of Mathematics,
Tereshchenkivska 3, Kiev, Ukraine\\
rybalkina\_t@ukr.net\quad sergeich@imath.kiev.ua}

\date{}

\maketitle

\begin{abstract}
We study systems of linear and
semilinear mappings considering them as
representations of a directed graph $G$
with full and dashed arrows: a
representation of $G$ is given by
assigning to each vertex a complex
vector space, to each full arrow a
linear mapping, and to each dashed
arrow a semilinear mapping of the
corresponding vector spaces. We extend
to such representations the classical
theorems by Gabriel about quivers of
finite type and by
Nazarova, Donovan, and Freislich about
quivers of tame types.
\medskip

Keywords: Linear and semilinear mappings,
quivers of finite and tame types,
classification
\medskip

AMS classification: 15A04, 15A21,
16G60
\end{abstract}

\section{Introduction}

We study systems of linear and
semilinear mappings on complex vector spaces.
A mapping ${\mathcal A}$ from a
complex vector space $U$ to a complex vector space $V$
is called \emph{semilinear} if
\begin{equation*}\label{lkhe}
{\mathcal A}(u+u')={\mathcal A}u+{\mathcal A}u',\qquad
{\mathcal A}(\alpha u)=\bar\alpha {\mathcal A}u
\end{equation*}
for all $u,u'\in U$ and $\alpha
\in\mathbb C$. We write ${\mathcal A}: U\to
V$ if ${\mathcal A}$ is a linear mapping
and ${\mathcal A}: U\dashrightarrow V$
(using a dashed arrow) if ${\mathcal A}$ is
a semilinear mapping.

We study systems of
linear and semilinear mappings
considering them as representations of
biquivers introduced by Sergeichuk \cite[Section
5]{ser_surv} (see also \cite{debora}); they generalize the
notion of representations of  quivers
introduced by Gabriel \cite{gab}.

\begin{definition}\label{def1}
\begin{itemize}
  \item A \emph{biquiver} is a
      directed graph $G$ with
      vertices $1,2,\dots,t$ and
      with full and dashed arrows;
      for
example,
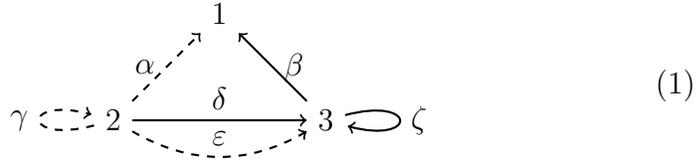
\begin{equation}\label{ksyq}
\begin{split}
\begin{tikzpicture}[->,thick,auto,
node distance=2 cm,xscale=2,yscale=2]
  \node (1) {$1$};
  \node(2) [below left of=1] {$2$};
  \node (3) [below right of=1] {$3$};
  \path
    (1)
    (2) edge [dashed] node [left] {${\alpha}$} (1)
        edge node {${\delta}$} (3)
        edge [bend right,dashed]
        node[above]%
        {${\varepsilon}$} (3)
        edge [dashed,loop left] node
          {${\gamma}$} (2)
    (3) edge [loop right] node
        {${\zeta}$} (3)
        edge node [right]
          {${\beta }$} (1)
    ;
\end{tikzpicture}
\end{split}
\end{equation}
  \item A \emph{representation}
      $\mathcal A$ of a biquiver
      $G$ is given by assigning to
      each vertex $v$ a complex
      vector space $\mathcal A_v$,
      to each full arrow $\alpha:
      u\longrightarrow v$ a linear
      mapping $\mathcal A_{\alpha
      }:\mathcal A_u\to\mathcal
      A_v$, and to each dashed
      arrow $\alpha:
      u\dashrightarrow v$ a
      semilinear mapping $\mathcal
      A_{\alpha }:\mathcal
      A_u\dashrightarrow\mathcal
      A_v$. The vector \[\dim
      \mathcal A:=(\dim \mathcal
      A_1,\dots,\dim \mathcal A_t)\]
      is called the
      \emph{dimension} of a
      representation $\mathcal A$.
For example, a representation
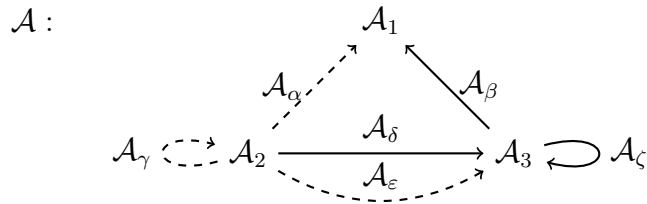
\begin{equation*}\label{2bf}
\begin{split}
\begin{tikzpicture}[->,thick,auto,
node distance=2.5cm,xscale=2,yscale=2]
 \node (1) {${\mathcal A}_1$};
  \node(2)[below left of=1]
           {${\mathcal A}_2$};
  \node (3) [below
  right of=1] {${\mathcal A}_3$};
 \node (4)[left of=1] {${\mathcal A}:$
\qquad\qquad\qquad\qquad\qquad\qquad};
\path (1) (2) edge [dashed] node
    [left] {${\mathcal A}_{\alpha}$}
    (1)
        edge node {${\mathcal A}_{\delta}$} (3)
        edge [bend right,dashed]
        node[above]%
        {${\mathcal A}_{\varepsilon}$} (3)
        edge [dashed,loop left] node
          {${\mathcal A}_{\gamma}$} (2)
    (3) edge [loop right] node
        {${\mathcal A}_{\zeta}$} (3)
        edge node [right]
          {${\mathcal A}_{\beta }$} (1)
    ;
\end{tikzpicture}
\end{split}
\end{equation*}
of \eqref{ksyq} is formed by
complex spaces ${\mathcal A}_1,{\mathcal
A}_2,{\mathcal A}_3$, linear mappings
${\mathcal A}_{\beta}$, ${\mathcal
A}_{\delta}$, ${\mathcal A}_{\zeta}$,
and semilinear mappings ${\mathcal
A}_{\alpha }$, ${\mathcal A}_{\gamma
}$, ${\mathcal A}_{\varepsilon }$.

  \item A \emph{morphism} $\mathcal
      F: {\mathcal A}\to{\mathcal B}$
      between representations $\mathcal
      A$ and ${\mathcal B}$ of a
      biquiver $G$ is a family  of
      linear mappings ${\mathcal F}_1:
      {\mathcal A}_1\to {\mathcal
      B}_1,\dots, {\mathcal F}_t: {\mathcal
      A}_t\to {\mathcal B}_t$ such that
for each arrow $\alpha$ from $u$ to $v$
the
diagram\\[-5mm]
\begin{equation}\label{1.2aa}
\begin{split}
\xymatrix{
 {\mathcal A_u}
\ar@{->}[r]^{\mathcal A_{\alpha }}
\ar@{->}[d]_{\mathcal F_{u}}
& {\mathcal A_v}
\ar@{->}[d]^{\mathcal F_{v}}
\\
{\mathcal B_u}
\ar@{->}[r]^{\mathcal B_{\alpha }}
& {\mathcal B_v}
 } \ \ \begin{matrix}
 \\[4mm]
 \text{if $\xymatrix{u
\ar@{->}[r]^{\alpha \ }&v}$\quad or} \\
       \end{matrix}
 \quad
\xymatrix{
 {\mathcal A_u}
\ar@{-->}[r]^{\mathcal A_{\alpha }}
\ar@{->}[d]_{\mathcal F_{u}}
& {\mathcal A_v}
\ar@{->}[d]^{\mathcal F_{v}}
\\
{\mathcal B_u}
\ar@{-->}[r]^{\mathcal B_{\alpha }}
& {\mathcal B_v}
 } \ \ \begin{matrix}
 \\[4mm]
 \text{if $\xymatrix{u
\ar@{-->}[r]^{\alpha \ }&v}$} \\
       \end{matrix}
\end{split}
\end{equation}
is commutative (i.e., $ {\mathcal
B}_{\alpha}{\mathcal F}_u={\mathcal
F}_v{\mathcal A}_{\alpha}$). We write $\mathcal A\simeq \mathcal B$ if $\mathcal A$ and  $ \mathcal B$ are isomorphic; i.e., if all ${\mathcal F}_i$ are bijections.
\end{itemize}
\end{definition}

For example, each cycle of linear
and semilinear mappings
\begin{equation*}\label{jsttw}
\begin{split}
{\mathcal A}: \qquad\xymatrix{
{V_1}
\ar@{-}@/^2pc/[rrrr]^{\mathcal A_t}
\ar@{-}[r]^{\ \ \mathcal A_1}&
V_2\ar@{-}[r]^{\mathcal A_2} &
{\ \dots\ }&{V_{t-1}}
\ar@{-}[l]_{\mathcal A_{t-2}}
\ar@{-}[r]^{\ \mathcal A_{t-1}\ \ }&{V_t}}
\end{split}
\end{equation*}
(in which each edge is a full or dashed
arrow $\longrightarrow$, $\longleftarrow$,
$\xymatrix{\ar@{-->}[r]&}$, or $\xymatrix{\ar@{<--}[r]&}$) is a representation of the biquiver
\begin{equation*}\label{aadw}
\begin{split}
C: \qquad\xymatrix{
{1}
\ar@{-}@/^2pc/[rrrr]^{\alpha _t}
\ar@{-}[r]^{\quad\alpha_1}&
2\ar@{-}[r]^{\alpha_2} &
{\ \dots\ }&{(t-1)}\ar@{-}[l]_{\alpha_{t-2}\ }
\ar@{-}[r]^{\alpha _{t-1}}&{t}\quad}
\end{split}
\end{equation*}
its representations were classified in \cite{debora1}.

Note that a biquiver without dashed
arrows is a quiver and its
representations are the quiver
representations. The quivers, for which
the problem of classifying their
representations does not contain the
problem of classifying pairs of
matrices up to similarity (i.e., the problem of classifying
representations of the quiver
 $\lefttorightarrow  \!\!1\!\!\righttoleftarrow$),
are called \emph{tame}
(this definition is informal; formal
definitions are given in \cite[Section
14.10]{gab_roi}). The list of all
tame quivers and the classification
of
their representations were obtained
independently by Donovan and Freislich
\cite{don1} and Nazarova \cite{naz}. We
extend their results to representations
of biquivers.

\section{Formulation of the main results} \label{kdtf}

The {\it direct sum} of representations
    ${\mathcal A}$ and ${\mathcal B}$ of a biquiver is
    the representation ${\mathcal
    A}\oplus{\mathcal B}$ formed by the
    spaces ${\mathcal A}_v\oplus {\mathcal
    B}_v$ and the mappings ${\mathcal
    A}_{\alpha}\oplus{\mathcal
    B}_{\alpha}$. A representation
    of nonzero dimension is {\it
    indecomposable} if it is not
    isomorphic to a direct sum of
    representations of nonzero
    dimensions.

By analogy with quiver representations,
we say that a biquiver is {\it
representation-finite} if it has only
finitely many nonisomorphic
indecomposable representations. A
biquiver is {\it wild} if the problem
of classifying its representations
contains the problem of classifying
matrix pairs up to similarity
transformations
\begin{equation*}\label{gpw}
(A,B)\mapsto (S^{-1}AS,S^{-1}AS),\qquad S\text{ is nonsingular},
\end{equation*}
otherwise the biquiver is {\it tame}. Clearly, each
representation-finite biquiver is tame.
The problem of classifying matrix pairs
up to similarity is the
problem of classifying representations
of the quiver
 $\lefttorightarrow
\!\!1\!\!\righttoleftarrow$; it
contains the problem of classifying
representations of each quiver (see
\cite{bel-ser_comp}) and so it is considered as
hopeless. An analogous statement for
representations of biquivers was proved
in \cite{debora}: the problem of
classifying representations of the
biquiver\!\!\!\!
\raisebox{-2.1mm}{\begin{tikzpicture}[->,thick,auto,node
distance=1cm,xscale=1.2,yscale=1.7]
  \node (1) {$1$};
\path
    (1)
edge [dashed,loop left] node
          {} (1)
edge [dashed,loop right] node
          {} (1)
          ;
\end{tikzpicture}}\!\!\!\!
contains the problem of classifying
representations of each biquiver.

The {\it Tits form} of a biquiver $G$
with vertices $1,\dots,t$ is the
integral quadratic form
\begin{equation*}\label{kur}
q_G(x_1,\dots,x_t):=x_1^2+\dots+x_t^2-
\sum x_ux_v,
\end{equation*}
in which the sum $\sum$ is taken over
all arrows $u\longrightarrow v$ and
$u\dashrightarrow v$ of the biquiver.

The following theorem extends Gabriel's
theorem {\cite{gab}} (see
also \cite[Theorem 2.6.1]{haz-kir}) to each biquiver $G$
and coincides with it if
$G$ is a quiver.

\begin{theorem}[proved in Section \ref{oir}]
\label{te1} Let
$G$ be a connected biquiver with
vertices $1,2,\dots,t$.
\begin{itemize}
  \item[\rm(a)] $G$ is
      representation-finite if and
      only if $G$ can be obtained from
      one of
      the Dynkin diagrams
\begin{equation}\label{dfe}
\begin{aligned}
&A_t\ \xymatrix@=10pt@R=0,5pt{
&&&&\\
*{\ci}\ar@{-}[r]&*{\ci}\ar@{-}[r]&
*{{\ci}\
\cdots\ {\ci}}
\ar@{-}[r]&*{\ci}\ar@{-}[r] &*{\ci}}
\qquad
&&
D_t\ \xymatrix@=10pt@R=0,5pt{&&&&*{\ci}\\
*{
\ \ci}\ar@{-}[r]&*{\ci}\ar@{-}[r]&
*{{\ci}\
\cdots\ {\ci}}
\ar@{-}[r]&*{\ci}\ar@{-}[ur]
 \ar@{-}[dr]&\\ &&&&*{\ci}}
          \\[-5pt]
&
\begin{matrix}
 \\
E_6
\end{matrix}
\ \xymatrix@=10pt{
&&*{\ci}\ar@{-}[d]&&\\
*{\ci}\ar@{-}[r]&*{\ci}\ar@{-}[r]&
*{\ci}
\ar@{-}[r]&*{\ci}\ar@{-}[r] &*{\ci}}
&&
\begin{matrix}
 \\
E_7
\end{matrix}
\ \xymatrix@=10pt{
&&*{\ci}\ar@{-}[d]&&\\
*{\ci}\ar@{-}[r]&*{\ci}\ar@{-}[r]&
*{\ci}
\ar@{-}[r]&*{\ci}\ar@{-}[r] &*{\ci}\ar@{-}[r] &*{\ci}}
          \\
&\begin{matrix}
 \\
E_8
\end{matrix}
\ \xymatrix@=10pt{
&&*{\ci}\ar@{-}[d]&&\\
*{\ci}\ar@{-}[r]&*{\ci}\ar@{-}[r]&
*{\ci}
\ar@{-}[r]&*{\ci}\ar@{-}[r] &*{\ci}\ar@{-}[r] &*{\ci}\ar@{-}[r] &*{\ci}}
\end{aligned}
\end{equation}
by replacing each edge with a full or   dashed arrow,
if and only if the Tits form $q_G$
      is positive definite.

  \item[\rm(b)] Let $G$ be
      representation-finite and let
      $ z=(z_1,\dots,z_t)$ be an
      integer vector with
      nonnegative components. There
      exists an indecomposable
      representation of dimension
      $z$ if and only if $q_G({
      z})=1$; this representation
      is determined by $z$ uniquely
      up to isomorphism.

\end{itemize}
\end{theorem}
Representations of
representation-finite quivers were
classified by Gabriel \cite{gab}; see
also \cite[Theorem 2.6.1]{haz-kir}.

The following theorem extends the
Donovan--Freislich--Nazarova theorem
\cite{don1,naz} (see also \cite[Chapter 2]{haz-kir}) to each biquiver $G$ and
coincides with it if $G$ is a quiver.

\begin{theorem}[proved in Section \ref{kur1}]
\label{te2} Let $G$
be a connected biquiver with
vertices $1,2,\dots,t$.
 \begin{itemize}
 \item[\rm(a)] $G$ is tame if and
     only if $G$ can be obtained from one of the
     Dynkin diagrams \eqref{dfe} or from one of the extended Dynkin diagrams
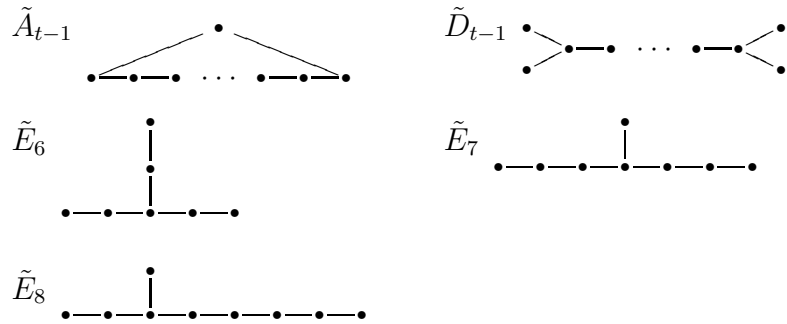
\begin{equation*}\label{dfse}
\begin{aligned}
&\tilde A_{t-1}\ \xymatrix@=10pt{
&&{\ci}\ar@{-}[lld]\ar@{-}[rrd]&&\\
*{\ci}\ar@{-}[r]&*{\ci}\ar@{-}[r]&
*{{\ci}\
\cdots\ {\ci}}
\ar@{-}[r]&*{\ci}\ar@{-}[r] &*{\ci}}
\qquad
&&
\tilde D_{t-1}\ \xymatrix@=10pt@R=0,5pt{*{\ci}&&&&*{\ci}\\
&*{\ci}\ar@{-}[r]\ar@{-}[lu]\ar@{-}[ld]&
*{{\ci}\
\cdots\ {\ci}}
\ar@{-}[r]&*{\ci}\ar@{-}[ur]
 \ar@{-}[dr]&\\ *{\ci}&&&&*{\ci}}
          \\[-5pt]
&
\begin{matrix}
 \\
\tilde E_6
\end{matrix}
\ \xymatrix@=10pt{
&&*{\ci}\ar@{-}[d]&&\\
&&*{\ci}\ar@{-}[d]&&\\
*{\ci}\ar@{-}[r]&*{\ci}\ar@{-}[r]&
*{\ci}
\ar@{-}[r]&*{\ci}\ar@{-}[r] &*{\ci}}
&&
\begin{matrix}
 \\
\tilde E_7
\end{matrix}
\ \xymatrix@=10pt{
&&&*{\ci}\ar@{-}[d]&&\\
*{\ci}\ar@{-}[r]&*{\ci}\ar@{-}[r]&*{\ci}\ar@{-}[r]&
*{\ci}
\ar@{-}[r]&*{\ci}\ar@{-}[r] &*{\ci}\ar@{-}[r] &*{\ci}}
          \\
&\begin{matrix}
 \\
\tilde E_8
\end{matrix}
\ \xymatrix@=10pt{
&&*{\ci}\ar@{-}[d]&&\\
*{\ci}\ar@{-}[r]&*{\ci}\ar@{-}[r]&
*{\ci}
\ar@{-}[r]&*{\ci}\ar@{-}[r] &*{\ci}\ar@{-}[r] &*{\ci}\ar@{-}[r] &*{\ci}\ar@{-}[r] &*{\ci}}
\end{aligned}
\end{equation*}
by replacing each edge with a full or   dashed arrow,
if and only if the Tits form $q_G$
     is positive semidefinite.

  \item[\rm(b)] Let $G$ be tame and
      let ${ z}=(z_1,\dots,z_t)$ be
      an integer vector with
      nonnegative components. There
      exists an indecomposable
      representation of dimension
      $z$ if and only if $q_G({
      z})=0$ or $q_G({z})=1$.
\end{itemize}
\end{theorem}

Representations of tame quivers were
classified by Donovan and Freislich \cite{don1}
and independently by Nazarova \cite{naz}.

The following theorem is a special case
of the Krull--Schmidt theorem for
additive categories \cite[Chapter I,
Theorem 3.6]{bas} (it holds for
representations of a biquiver since
they form an additive category in which
all idempotents split).

\begin{theorem}\label{jhp}
Each representation of a biquiver is
isomorphic to a direct sum of
indecomposable representations. This
direct sum is uniquely determined, up
to permutations and isomorphisms of
direct summands, which means that if \[
\mathcal A_1\oplus\dots\oplus\mathcal
A_r\simeq \mathcal B_1\oplus\dots\oplus
\mathcal B_s,\] in which all $\mathcal
A_i$ and $\mathcal B_j$ are
indecomposable representations, then
$r=s$ and all $\mathcal A_i\simeq
\mathcal B_i$ after a suitable
renumbering of $\mathcal
A_1,\dots,\mathcal A_r$.
\end{theorem}

\section{Matrix representations of
biquivers}

Let us recall some elementary facts about semilinear mappings.

We denote by $[v]_e$ the coordinate
vector of $v$ in a basis
$e_1,\dots,e_n$, and by $S_{e\to e'}$
the transition matrix from a basis
$e_1,\dots,e_n$ to a basis
$e'_1,\dots,e'_n$.  If $A=[a_{ij}]$
then $\bar A:=[\bar a_{ij}]$.

Let ${\mathcal A}: U\dashrightarrow V$  be
a semilinear mapping. We say that an
$m\times n$ matrix ${\mathcal A}_{fe}$ is
the \emph{matrix of $\mathcal A$} in bases
$e_1,\dots,e_n$ of $U$ and
$f_1,\dots,f_m$ of $V$ if
\begin{equation}\label{feo2}
[{\mathcal A}u]_f=\overline{{\mathcal A}_{fe}[u]_e}\qquad
\text{for all }u\in U.
\end{equation}
Therefore, the columns of ${\mathcal
A}_{fe}$ are $\overline{[{\mathcal
A}e_1]_f}, \dots, \overline{[{\mathcal
A}e_n]_f}$. We write ${\mathcal A}_{e}$
instead of ${\mathcal A}_{ee}$ if $U=V$.

If $e'_1,\dots,e'_n$ and
$f'_1,\dots,f'_m$ are other bases of
$U$ and $V$, then
\begin{equation}\label{swk3}
{\mathcal A}_{f'e'}=\bar S_{f\to f'}^{-1}{\mathcal A}_{fe}
S_{e\to e'}
\end{equation}
since the right hand matrix satisfies
\eqref{feo2} with $e',f'$ instead of
$e,f$:
\[
\overline{\bar S_{f\to f'}^{-1}{\mathcal A}_{fe}
S_{e\to e'}[v]_{e'}}=
S_{f\to f'}^{-1}\overline{{\mathcal A}_{fe}
[v]_{e}}=S_{f\to f'}^{-1}[{\mathcal A}
v]_f=[{\mathcal A}
v]_{f'}
\]

\begin{lemma}\label{kowq}
Let $U$, $V$, and $W$ be vector spaces
with bases $e_1,e_2,\dots,$
$f_1,f_2,\dots,$, and $g_1,g_2,\dots$.
\begin{itemize}
  \item[\rm (a)] The composition of
      a linear mapping ${\mathcal
      A}: U\to V$ and a semilinear
      mapping ${\mathcal B}:
      V\dashrightarrow W$ is the
      semilinear mapping with
      matrix
\begin{equation}\label{lif}
(\mathcal B \mathcal
      A)_{ge}=\mathcal B_{gf}
      \mathcal A_{fe}
\end{equation}

 \item[\rm (b)] The composition of
     a semilinear mapping
     ${\mathcal A}:
     U\dashrightarrow V$ and a
     linear mapping ${\mathcal B}:
     V\to W$ is the semilinear
     mapping with matrix
\begin{equation}\label{kif}
(\mathcal B \mathcal
A)_{ge}=\overline{\mathcal B}_{gf} \mathcal
A_{fe}
\end{equation}

\end{itemize}
\end{lemma}

\begin{proof}
The identity \eqref{lif} follows from
observing that ${\mathcal AB}$ is a
semilinear mapping and
\[
[({\mathcal B}{\mathcal A})u]_{ge}=
[{\mathcal B}({\mathcal A}u)]_{ge}=
\overline{{\mathcal B}_{gf}[{\mathcal
A}u]_{fe}}= \overline{({\mathcal
B}_{gf}{\mathcal A}_{fe})[u]_{e}}
\]
for each $u\in U$. The identity
\eqref{kif} follows from observing that
${\mathcal AB}$ is a semilinear mapping
and
\[
[({\mathcal B}{\mathcal A})u]_{ge}=
[{\mathcal B}({\mathcal A}u)]_{ge}=
{\mathcal B}_{gf}[{\mathcal
A}u]_{fe}= {\mathcal
B}_{gf}\overline{{\mathcal A}_{fe}[u]_{e}}=
\overline{(\overline{{\mathcal
B}}_{gf}{\mathcal A}_{fe})[u]_{e}}
\]
for each $u\in U$.
\end{proof}

Let $\mathcal A: V\dashrightarrow V$ be
a semilinear mapping; let ${\mathcal
A}_{e}$ and ${\mathcal A}_{e'}$ be its
matrices in bases $e_1,\dots,e_n$ and
$e'_1,\dots,e'_n$ of $V$. By
\eqref{swk3},
\begin{equation}\label{swk4}
{\mathcal A}_{e'}=\bar S_{e\to e'}^{-1}{\mathcal A}_{e}
S_{e\to e'}
\end{equation}
and so ${\mathcal A}_{e'}$ and ${\mathcal
A}_{e}$ are consimilar: recall that two
matrices $A$ and $B$ are
\emph{consimilar} if there exists a
nonsingular matrix $S$ such that $\bar
S^{-1}AS=B$; a canonical form of a square complex matrix under consimilarity is given in \cite[Theorem 4.6.12]{j-h}.

Each representation $\mathcal A$ of a
biquiver $G$ can be given by the set
$A$ of matrices ${A}_{\alpha }$ of its
mappings ${\mathcal A}_{\alpha }$ in fixed
bases of the spaces ${\mathcal
A}_1,\dots,{\mathcal A}_t$. Changing the
bases, we can reduce ${\mathcal A}_{\alpha
}$ by transformations
$S^{-1}_v{A}_{\alpha}{S}_u$ if $\alpha:
u\longrightarrow v$ and
$\bar{S}^{-1}_v{A}_{\alpha}{S}_u$ if
$\alpha: u\dashrightarrow v$, in which
$S_1,\dots,S_t$ are the transition
matrices, which reduces the problem of classifying
representations of $G$ up to isomorphism
to the problem of classifying
the sets $A$ up to these
transformations. This leads to the
following definition.

\begin{definition}\label{def2}
Let $G$ be a
      biquiver with vertices
      $1,\dots,t$.
\begin{itemize}
  \item A \emph{matrix
      representation $A$ of
      dimension $(d_1,\dots,d_t)$}
      of $G$ is given by assigning
      an ${d_v\times d_u}$ complex
      matrix ${A}_{\alpha}$ to each
      arrow
      $\alpha:u\longrightarrow v$
or $u\dashrightarrow v$.

  \item Two matrix representations
      $A$ and $B$ of $G$ are
      \emph{isomorphic} if there
      exist nonsingular matrices
      $S_1,\dots, S_t$ such that
\begin{equation}\label{ljt}
{B}_{\alpha}=
  \begin{cases}
    S^{-1}_v{A}_{\alpha}{S}_u &
\hbox{for every full arrow $\alpha:
u\longrightarrow v$,} \\
\bar{S}^{-1}_v{A}_{\alpha}{S}_u &
\hbox{for every dashed arrow $\alpha:
u\dashrightarrow v$.}
  \end{cases}
\end{equation}
\end{itemize}
\end{definition}

Each matrix representation $A$ of
dimension $d=(d_1,\dots,d_t)$ can be
identified with the representation
$\mathcal A$ from the Definition
\ref{def1}, whose vector spaces have
the form
\[
{\mathcal A}_v={\mathbb
C}\oplus\dots\oplus {\mathbb C}\quad(\text{$d_v$ summands})
\]
for all vertices $v$ and the linear or
semilinear mappings ${\mathcal A}_{\alpha
}$ are defined by the matrices
$A_{\alpha }$. A \emph{morphism} ${\mathcal
F}:{\mathcal A}\to{\mathcal B}$ of such
representations of dimensions $d$ and
$d'$ is given by a set of matrices
${F}_i\in\mathbb C^{d_i'\times d_i}$
such that
\begin{equation}\label{l,t}
{B}_{\alpha}{F}_u=
  \begin{cases}
    F_v{A}_{\alpha} & \hbox{for every arrow $\alpha:
u\longrightarrow v$,} \\
\overline{F}_v{A}_{\alpha} &
\hbox{for every arrow $\alpha:
u\dashrightarrow v$}
  \end{cases}
\end{equation}
(these equalities are obtained from
\eqref{1.2aa} in view of Lemma
\ref{kowq}.) In particular, if ${\mathcal
F}$ is isomorphism, then we can put
$S_v:=F^{-1}_v$ for all vertices $v$
and rewrite \eqref{l,t} in the form
\eqref{ljt}.

 By \eqref{swk4} and \eqref{ljt},
\begin{quote}
\emph{two matrix representations are
isomorphic if and only if they give the
same representation $\mathcal A$ but in
possible different bases.}
\end{quote}

Denote by $M(G)$ the set of matrix
representations of a biquiver $G$.

\section{Proof of Theorem \ref{te1}}
\label{oir}

For each biquiver $G$ and its vertex
$u$, we denote by $G^u$ the biquiver
obtained from $G$ by replacing all
arrows $u\longrightarrow v$ and
$v\longrightarrow u$ for each $v\ne u$
by $u\dashrightarrow v$ and
$v\dashrightarrow u$, and vice versa.
For example,
\newcommand{\examp}[7]
{\raisebox{8em}{$#7$}\text{{\begin{tikzpicture}[->,thick,auto,
node
distance=1.7cm,xscale=2.5,yscale=2.5]
  \node (0) {$u$};
  \node(1) [above of=0] {$1$};
  \node (2) [left of=0] {$2$};
  \node (3) [below of=0] {$3$};
  \node (4) [right of=0] {$4$};
 \path
(1)
     edge [dashed] node [left]
    {${\scriptstyle #1}$} (0)
(2)
     edge node [above] {${\scriptstyle #2}$} (0)
(0)
     edge [dashed] node [right] {${\scriptstyle #3}$} (3)
     edge node [below] {${\scriptstyle #4}$} (4)
     edge [in=23,out=67,loop,dashed] node
            {${\scriptstyle #6}$} (0)
     edge [in=203,out=247,loop] node
        {${\scriptstyle #5}$}(0)
;
\end{tikzpicture}
}}}
\newcommand{\exam}[7]
{\raisebox{8em}{$#7$}\text{{\begin{tikzpicture}[->,thick,auto,
node
distance=1.7cm,xscale=2.5,yscale=2.5]
  \node (0) {$u$};
  \node(1) [above of=0] {$1$};
  \node (2) [left of=0] {$2$};
  \node (3) [below of=0] {$3$};
  \node (4) [right of=0] {$4$};
 \path
(1)
     edge node [left] {${\scriptstyle #1}$} (0)
(2)
     edge [dashed] node [above] {${\scriptstyle #2}$} (0)
(0)
     edge node [right] {${\scriptstyle #3}$} (3)
     edge [dashed] node [below] {${\scriptstyle #4}$} (4)
     edge [in=23,out=67,loop,dashed] node
        {${\scriptstyle #6}$} (0)
     edge [in=203,out=247,loop] node
        {${\scriptstyle #5}$}(0)
;
\end{tikzpicture}
}}}
\begin{equation}\label{hdd}
\begin{split}
\examp{\alpha_1}{\alpha_2}{\alpha_3}
{\alpha_4}{\alpha_5}{\alpha_6}{G:}
   \qquad
\exam{\alpha_1}{\alpha_2}{\alpha_3}
{\alpha_4}{\alpha_5}{\alpha_6}{G^u:}
\end{split}
\end{equation}
We say that $G^u$ is
obtained from $G$ by \emph{conjugation}
at the vertex $u$.
For each $A\in M(G)$, define $ A^u\in
M(G^u)$ as follows:
\begin{equation*}\label{kurw}
 A^u_{\alpha }:=
             \begin{cases}
A_{\alpha } & \hbox{if $\alpha$ does not start at $u$},
      \\
\bar{A}_{\alpha }
& \hbox{if $\alpha$ starts at $u$.}
             \end{cases}
\end{equation*}
We say that $A^u$ is obtained from $A$
by \emph{conjugation} at the vertex
$u$.

\begin{lemma}\label{kow2}
Let ${ A},{B}\in M(G)$ and let $u$ be
any vertex of $G$. Then ${A}\simeq{B}$
if and only if ${A}^u\simeq{B}^u$.
\end{lemma}

\begin{proof}
It suffices to prove that
\begin{equation}\label{ahd}
\parbox{22em}
{if $A \simeq B$ via $S_1,\dots,S_t$
(see \eqref{ljt}), then $A^u \simeq
B^u$ via $R_1,\dots,R_t$, in which
$R_v:=S_v$ if $v\ne u$ and $R_u:=\bar
S_u$.}
\end{equation}
Moreover, it suffices to prove
\eqref{ahd} for matrix representations
of the biquiver $G$ defined in
\eqref{hdd}, which contains all
possible arrows from the vertex $u$ and
to the vertex $u$.

Let us consider an arbitrary matrix
representation $A$ of $G$ and the
corresponding matrix representation
$A^u$ of $G^u$:
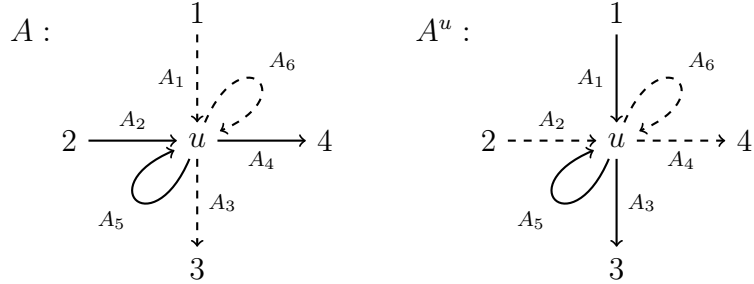
\[
\examp{A_1}{A_2}{A_3}
{A_4}{A_5}{A_6}{A:}
   \qquad
\exam{A_1}{A_2}{A_3}
{A_4}{A_5}{A_6}{A^u:}
\]

Let $B$ be any matrix representation of
$G$ that is isomorphic to $A$ via
$S_1,\dots S_4,S_u$. Then $B$ and the
corresponding matrix representation
$B^u$ of $G^u$ have the form:
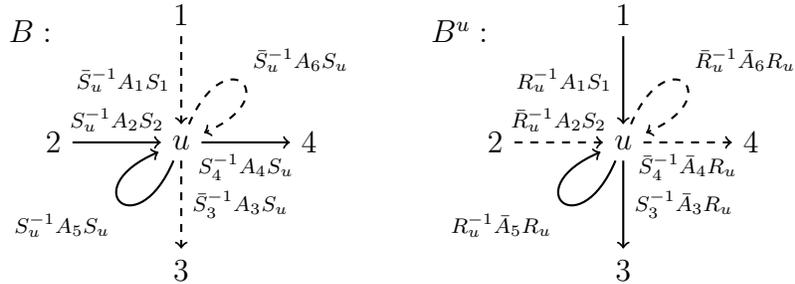
\[
\examp{\bar S_u^{-1}A_1S_1}{S_u^{-1}A_2S_2}
{\bar S_3^{-1}A_3S_u}
{S_4^{-1}A_4S_u}{S_u^{-1}A_5S_u}{\bar S_u^{-1}A_6S_u}
{B:\!\!\!\!\!\!\!\!\!}
   \qquad
\exam{R_u^{-1}A_1S_1}{\bar R_u^{-1}A_2S_2}
{S_3^{-1}\bar A_3R_u}
{\bar S_4^{-1}\bar A_4R_u}
{R_u^{-1}\bar A_5R_u}{\bar R_u^{-1}\bar A_6R_u}
{B^u:\!\!\!\!\!\!\!\!\!}
\]
in which $R_u$ is defined by
\eqref{ahd}.

Therefore, $B^u$ is isomorphic to $A^u$
via $S_1,\dots S_4,R_u$, which proves
\eqref{ahd}.
\end{proof}

\begin{proof}[Proof of Theorem
\ref{te1}] Let $G$ be a connected
bigraph with $t$ vertices.

(a) Suppose first that $G$ is a tree.
Let us prove that by a sequence of
conjugations we can transform $G$ to
the quiver $Q(G)$ obtained from $G$ by
replacing each dashed arrow
$v\dashrightarrow w$ with the full
arrow $ v\longrightarrow w$.

Let $w$ be a pendant vertex of $G$
(i.e., a vertex of degree 1). Let
$\alpha$ be the arrow for which $w$ is
one of its vertices. Denote by
$G\setminus\alpha$ the biquiver
obtained from $G$ by deleting $w$ and
$\alpha $. Reasoning by induction on
the number of vertices, we assume that
$G\setminus\alpha$ can be transformed
to $Q(G\setminus\alpha)$ by a sequence
of conjugations. The same sequence of
conjugations  transforms $G$ to some
biquiver $G'$ in which only the arrow
that is obtained from $\alpha $ can be
dashed. If it is dashed, we make it
full by conjugation of $G'$ at the vertex
$w$ and obtain $Q(G)$. Theorem
\ref{te1} holds for $Q(G)$ by Gabriel's
theorem {\cite{gab}}. Lemma \ref{kow2}
ensures that Theorem \ref{te1} holds
for $G$ too.

Suppose now that $G$ is not a tree.
Then $G$ contains a cycle $C$ that up
to renumbering of vertices of $G$ has
the form
\begin{equation}\label{aay}
\begin{split}
C: \qquad\xymatrix{
{1}
\ar@{-}@/^2pc/[rrrr]^{\alpha _r}
\ar@{-}[r]^{\quad\alpha_1}&
2\ar@{-}[r]^{\alpha_2} &
{\ \dots\ }&{(r-1)}\ar@{-}[l]_{\alpha_{r-2}\ }
\ar@{-}[r]^{\alpha _{r-1}}&{r}\quad}
\end{split}
\end{equation}
 If $r>1$ and the
sequence of arrows $\alpha
_1,\dots,\alpha _{r-1}$ contains a
dashed arrow, then we take
the first dashed arrow $\alpha _{\ell}$ and make it
full by conjugation of $G$ at the
vertex $\ell+1$. Repeating this
procedure, we make all arrows $\alpha
_1,\dots,\alpha _{r-1}$ full.

For each
$n\times n$ matrix $M$, let us construct the
matrix representation $P(M)$ of $G$ by
assigning $I_n$ to each of the arrows
$\alpha _1,\dots,\alpha _{r-1}$ (if
$r>1$), $M$ to $\alpha _r$, and $0_n$
to the other arrows. It is easy to see
that $P(M)\simeq P(N)$ if and only if either
$\alpha _r$ is full and $M$ is similar
to $N$, or $\alpha _r$ is dashed and
$M$ is consimilar to $N$ (see \eqref{swk4}). The Jordan
canonical form and a canonical form under consimilarity  \cite[Theorem 4.6.12]{j-h}
ensure that $G$ is of infinite type.
Since $G$ contains a cycle, it cannot
be obtained by directing edges in one
of the Dynkin diagrams \eqref{dfe}, and
so $q_G$ is not positive definite by
Gabriel's theorem {\cite{gab}}.

(b) If $G$ is of finite type, then $G$
is a tree. By the part (a) of the
proof, $G$ can be transformed to the
quiver $Q(G)$ by a sequence of
conjugations. By Lemma \ref{kow2}, this
sequence of conjugations transforms all
indecomposable representations of $G$
to all indecomposable representations
of $Q(G)$, and nonisomorphic
representations are transformed to
nonisomorphic representations. This
proves (b) for $G$ since (b) holds for
$Q(G)$.
\end{proof}

\section{Proof of Theorem \ref{te2}}\label{kur1}

\begin{lemma}\label{kjr}
The problem of classifying
representations of each of the
biquivers
\begin{equation}\label{fuv}
\begin{split}
\text{
\begin{tikzpicture}[->,thick,auto,node
distance=1.5cm,xscale=1.7,yscale=3]
  \node (1) {$1$};
  \node(2) [right of=1] {$2$};
\node (0)[left of=1]
{$G_1:$\hspace*{2cm}};
\path
    (1)
edge [dashed,loop left] node
          {${\alpha _1}$} (1)
          edge node [above]
          {${\alpha }$} (2);
\end{tikzpicture}\hspace*{1cm}
\begin{tikzpicture}[->,thick,auto,node
distance=1.5cm,xscale=1.7,yscale=3]
  \node (1) {$1$};
  \node(2) [right of=1] {$2$};
\node (0)[left of=1]
{$G_2:$\hspace*{2cm}};
\path
    (1)
edge [dashed,loop left] node
          {${\alpha _1}$} (1)
(2) edge node [above]
          {${\alpha }$} (1);
\end{tikzpicture}
}
\end{split}
\end{equation}
contains the problem of classifying
representations of the biquiver
\begin{equation}\label{mos}
\begin{split}\text{
\begin{tikzpicture}[->,thick,auto,node
distance=2.5cm,xscale=1.7,yscale=3]
  \node (1) {$1$};
  \node (0)[left of=1]
{$G_3:$};
\path
    (1)
edge [loop left] node
          {${\alpha _1}$} (1)
edge [dashed,loop right] node
          {${\alpha _2}$} (1)
          ;
\end{tikzpicture}
}\end{split}
\end{equation}
\end{lemma}

\begin{proof}
The problem of classifying
representations of the biquivers
\eqref{fuv} is the problem of
classifying matrix pairs up to
transformations
\begin{align}\label{jre}
(M,N)&\mapsto (\bar S^{-1}MS,R^{-1}NS),\\
(M,N)&\mapsto (\bar S^{-1}MS,S^{-1}NR),
\label{jre1}
\end{align}
respectively.

Let us consider $G_1$. Let
\begin{equation*}\label{gec}
M:=\begin{bmatrix}
     0 & P \\
     I & Q
   \end{bmatrix},\qquad
M':=\begin{bmatrix}
     0 & P' \\
     I & Q'
   \end{bmatrix},\qquad
 N:=\begin{bmatrix}
     0 & I
   \end{bmatrix},
\end{equation*}
in which all blocks are $n$-by-$n$. Let
$(M,N)$ be reduced to $(M',N)$ by
transformations \eqref{jre}; i.e.,
there exist nonsingular $S$ and $R$
such that
\begin{equation}\label{fsy}
MS=\bar SM',\qquad NS=RN.
\end{equation}
By the second equality in \eqref{fsy},
$S$ has the form
\[
S=\begin{bmatrix}
     S_1 & S_2 \\
     0 & R
   \end{bmatrix}.
\]
Equating the 1,1 blocks in the first
equality in \eqref{fsy} gives $S_2=0$;
equating the 2,1 blocks gives
$S_1=\bar R$; equating the 1,2 and
2,2 blocks gives
\begin{equation*}\label{tvs}
PR=RP',\qquad  QR=\bar RQ'.
\end{equation*}
Therefore, $(M,N)$ and $(M',N)$
define isomorphic representations of
$G_1$ if and only if $(P,Q)$ and
$(P',Q')$ define isomorphic
representations of \eqref{mos}, and so
the problem of classifying
representations of $ G_1$ contains the
problem of classifying representations
of \eqref{mos}.

Let us consider $G_2$. Taking
$N:=\left[\begin{smallmatrix}I\\0
\end{smallmatrix}\right]$ and reasoning
as for $G_1$, we prove that
if $(M,N)$ is reduced to $(M',N)$
by transformations \eqref{jre1}; i.e.,
there exist nonsingular $S$ and $R$
such that $MS=\bar SM'$ and
$NR=SN$, then $S$ is upper block
triangular and so $(P,Q)$ and $(P',Q')$
define isomorphic representations of
\eqref{mos}.
\end{proof}

\begin{lemma}\label{kjrd}
The problems of classifying
representations of the
biquiver $G_3$ defined in
\eqref{mos} and the
biquiver
\begin{equation*}\label{mos1}
\begin{split}\text{
\begin{tikzpicture}[->,thick,auto,node
distance=2.5cm,xscale=1.7,yscale=3]
  \node (1) {$1$};
  \node (0)[left of=1]
{$G_4:$};
\path
    (1)
edge [dashed,loop left] node
          {${\alpha _1}$} (1)
edge [dashed,loop right] node
          {${\alpha _2}$} (1)
          ;
\end{tikzpicture}
}\end{split}
\end{equation*}
contain the problem of classifying
matrix pairs up to similarity.
\end{lemma}

\begin{proof}
The problems of classifying
representations of the biquivers $G_3$
and $G_4$ are the problems of classifying
matrix pairs up to transformations
\begin{align}\label{jre3}
(M,N)&\mapsto (S^{-1}MS,\bar S^{-1}NS),\\
(M,N)&\mapsto (\bar S^{-1}MS,\bar S^{-1}NS),
\label{jre4}
\end{align}
respectively.

Let us consider $G_3$. Let
\begin{equation*}\label{jgec}
M:=\begin{bmatrix}
     0 & I \\
     0 & 0
   \end{bmatrix},\qquad
N:=\begin{bmatrix}
     P & 0 \\
     0 & Q
   \end{bmatrix},\qquad
N':=\begin{bmatrix}
     P' & 0 \\
     0 & Q'
   \end{bmatrix},
   \end{equation*}
in which all blocks are $n$-by-$n$. Let
$(M,N)$ be reduced to $(M,N')$ by
transformations \eqref{jre3}; i.e.,
there exists a nonsingular $S$ such
that
\begin{equation}\label{fsyw}
MS=SM,\qquad NS=\bar S N'.
\end{equation}
By the first equality in \eqref{fsyw},
$S$ has the form
\[
S=\begin{bmatrix}
     S_1 & S_2 \\
     0 & S_1
   \end{bmatrix}.
\]
Equating the 1,1 and 2,2 blocks in
the second equality in \eqref{fsyw}
gives
\begin{equation*}\label{tvsd}
\bar S_1^{-1}PS_1=P',
\qquad \bar S_1^{-1}QS_1=Q'.
\end{equation*}
Therefore, $(M,N)$ and $(M,N')$ define
isomorphic representations of $G_3$ if
and only if $(P,Q)$ and $(P',Q')$
define isomorphic representations of
$G_4$, and so the problem of
classifying representations of $ G_3$
contains the problem of classifying
representations of $G_4$.

Let us consider $G_4$. Let
\begin{equation*}\label{jgj}
M:=\begin{bmatrix}
     0 & I&0&0 \\
     0 & 0&I&0\\ 0 & 0&0&I\\ 0 & 0&0&0
   \end{bmatrix},\ \
N:=\begin{bmatrix}
 0 & 0&0&0 \\
     P & 0&0&0\\ 0 & 0&0&0\\ 0 & 0&Q&0
   \end{bmatrix},\ \
N':=\begin{bmatrix}
 0 & 0&0&0 \\
     P' & 0&0&0\\ 0 & 0&0&0\\ 0 & 0&Q'&0
   \end{bmatrix}\end{equation*}
in which all blocks are $n$-by-$n$. Let
$(M,N)$ be reduced to $(M,N')$ by
transformations \eqref{jre4}; i.e.,
there exists a nonsingular $S$ such
that
\begin{equation}\label{fstw}
MS=\bar SM,\qquad NS=\bar S N'.
\end{equation}
By the first equality in \eqref{fstw},
$S$ has the form
\[
S=\begin{bmatrix}
     S_1 & S_2&S_3&S_4 \\
     0 &\bar S_1 &\bar S_2&\bar S_3\\
0&0&S_1 & S_2 \\0& 0&0&\bar S_1 \\
\end{bmatrix}.
\]
Equating the 2,1 and 4,3 blocks in
the second equality in \eqref{fstw}
gives
\[
S_1^{-1}PS_1=P',
\qquad S_1^{-1}QS_1=Q'.
\]
Therefore, $(M,N)$ and $(M,N')$ define
isomorphic representations of $G_4$ if
and only if $(P,Q)$ and $(P',Q')$ are
similar, and so the problem of
classifying representations of $ G_4$
is wild.
\end{proof}

\begin{proof}[Proof of Theorem
\ref{te2}]

 (a) Suppose first that $G$ is a tree.
Reasoning as in the proof of Theorem
\ref{te1}, we transform $G$ to the
quiver $Q(G)$ by a sequence of
conjugations. Theorem \ref{te2} holds
for $Q(G)$ by the
Donovan--Freislich--Nazarova theorem
{\cite{don1,naz}}. Lemma \ref{kow2}
ensures that Theorem \ref{te2} holds
for $G$ too.

Suppose now that $G$ is not a tree.
Then $G$ contains a cycle $C$ that up
to renumbering of vertices of $G$ has
the form \eqref{aay} in which $r\ge 1$
and each edge is a full or dashed
arrow.

If $G=C$, then $G$ is of tame type; all
its representations were classified in
\cite{debora1}.

Let us suppose that $G\ne C$ and prove that
$G$ is of wild type.  The biquiver $G$
contains a biquiver $C'$ obtained by
adjoining to $C$ an edge $\alpha :u
\to v$ or $v\to u$, in which $u\in\{1,\dots,r\}$,
we suppose that $u=1$. If $C'$ is of
wild type, then $G$ is of wild type
too: we can identify all
representations of $C'$ with those
representations of $G$, in which the
vertices outside of $C'$ are assigned
by $0$; two representations of $C'$ are
isomorphic if and only if the
corresponding representations of $G$
are isomorphic. Further we suppose that
$G=C'$.

Reasoning as in the proof of Theorem
\ref{te1}(a), we can transform the subbiquiver
\[
\xymatrix{
{2}\ar@{-}[r]^{\alpha_2}&
{3}\ar@{-}[r]^{\alpha_3}&
\ \dots\ \ar@{-}[r]^{\alpha_{r-1}}&r
\ar@{-}[r]^{\alpha_{r}}&1}
\]
of $G$ to a quiver by a sequence of
conjugations in some of the vertices
$3,4,\dots,1$. Thus, we can suppose
that the arrows $\alpha _2,\dots,\alpha
_r$ of $G$ are full arrows.

Suppose first that $v$ is not a vertex
of $C$. If  $\alpha :u \lin v$ is
dashed, we make it full by conjugation
at $v$. If $\alpha _1$ is a full arrow,
then $G$ is a quiver of wild type by
the Donovan--Freislich--Nazarova
theorem. Thus, we can suppose that
$\alpha _1$ is a dashed arrow. Let
$\ell\in\{1,2\}$ be such that $\alpha $
in $G$ has the same direction as
$\alpha_1$ in $G_{\ell}$ defined
in \eqref{fuv}. The biquiver $G$ is of
wild type since $G_{\ell}$ is of wild
type and each matrix representation $A$
of $G_{\ell}$ can be identified with
the matrix representation of $G$
obtained from $A$ by assigning the
identity matrix to $\alpha
_2,\dots,\alpha _r$; two
representations of $G_{\ell}$ are
isomorphic if and only if the
corresponding representations of $G$
are isomorphic.

Suppose now that $v$ is a vertex of
$C$. If $\alpha $ and $\alpha _1$ are
full arrows, then $G$ is a quiver of
wild type by the
Donovan--Freislich--Nazarova theorem.
Let $\alpha$ or $\alpha _1$ be a dashed
arrow. Denote by $G'$ the biquiver
obtained from $G$ by deleting its
arrows $\alpha _2,\dots,\alpha _r$ and
its vertices $2,3,\dots,r-1$, and by
identifying the vertices $1$
and $r$. By Lemma \ref{kjrd}, $G'$ is
of wild type. Hence, $G$ is of wild
type too since each matrix
representation $A$ of $G'$ can be
identified with the matrix
representation of $G$ obtained from $A$
by assigning the identity matrix to
$\alpha _2,\dots,\alpha _r$; two
representations of $G'$ are isomorphic
if and only if the corresponding
representations of $G$ are isomorphic.

(b) Let $G$ be tame. Then $G$ is a tree or cycle. If $G$ is a tree, then it can be transformed to the
quiver $Q(G)$ by a sequence of
conjugations. By Lemma \ref{kow2}, this
sequence of conjugations transforms all
indecomposable representations of $G$
to all indecomposable representations
of $Q(G)$, and nonisomorphic
representations are transformed to
nonisomorphic representations. This
proves (b) for $G$ since (b) holds for quivers by the
Donovan--Freislich--Nazarova theorem. If $G$ is a cycle then (b) follows from the classification of its representations given in \cite{debora1}.
\end{proof}

\bibliographystyle{amsalpha}

\begin{thebibliography}{99}


\bibitem{bas} H. Bass, \textit{Algebraic
 $K$-theory}, Benjamin, New York,
 1968.


\bibitem{bel-ser_comp} G.R. Belitskii and
    V.V. Sergeichuk, \textit{Complexity of matrix
    problems}, Linear Algebra
    Appl. \textbf{361} (2003), 203--222.

\bibitem{don1} P. Donovan and M.R.
    Freislich, \textit{The
    representation theory of finite
    graphs and associated algebras},
    Carleton Math. Lecture Notes, 5,
    Carleton University,  Ottawa, 1973.

\bibitem{debora} D. Duarte de Oliveira,
    R.A. Horn, T. Klimchuk, and V.V.
    Sergeichuk, \textit{Remarks on the
    classification of a pair of
    commuting semilinear operators},
Linear
    Algebra Appl. \textbf{436} (2012),
    3362--3372.

\bibitem{debora1} D. Duarte de Oliveira,
    V. Futorny, T. Klimchuk, D. Kovalenko, and V.V. Sergeichuk,
\textit{Cycles  of  linear  and  semilinear  mappings}, Linear Algebra Appl. \textbf{438} (2013), 3442--3453.

\bibitem{haz-kir} M. Hazewinkel, N.
    Gubareni, and V.V.
    Kirichenko. \textit{Algebras, Rings
    and Modules}. Vol.~2. Springer,
    2007.

\bibitem{j-h} R.A. Horn and C.R. Johnson,
    \textit{Matrix Analysis}, 2nd ed., Cambridge
    University Press, Cambridge, 2013.


\bibitem{gab} P. Gabriel, \textit{Unzerlegbare
    Darstellungen I}, {Manuscripta
    Math.} \textbf{6} (1972), 71--103.

\bibitem{gab_roi} P. Gabriel and A.V.
    Roiter,
    \textit{Representations of
    finite-dimensional algebras},
    Encyclopaedia of Math. Sci. Vol. 73
    (Algebra VIII), Springer-Verlag,
    1992.

\bibitem{naz} L.A.
    Nazarova, \textit{Representations of
    quivers of infinite type},
    Math. USSR Izv. \textbf{7} (1973), 749--792.

\bibitem{ser_surv}
  V.V. Sergeichuk, \textit{Linearization method in
classification problems of linear
algebra}, S\~ ao Paulo J. Math. Sci. \textbf{1}
(no. 2) (2007), 219--240.
\end{thebibliography}

\end{document}